\documentclass[reqno,11pt]{amsart}
\usepackage{graphicx, amsmath, amssymb, amscd, mathtools, hyperref, amsthm, euscript, 
amsfonts,bm,color,bbm}  
\DeclareMathAlphabet{\mathpzc}{OT1}{pzc}{m}{it}

\newtheorem{thm}[equation]{Theorem}\newtheorem{Thm}[equation]{Theorem}

\newtheorem{Rmk}[equation]{Remark}

\newtheorem{prop}[equation]{Proposition}

\newtheorem{cor}[equation]{Corollary}
\newtheorem{lem}[equation]{Lemma}
\newtheorem{lemma}[equation]{Lemma}

\numberwithin{equation}{section}

\numberwithin{equation}{section}

\newcommand{\be}{begin{equation}}

\newcommand{\bH}{\mathbb H}

\newcommand{\e}{\epsilon}

\newcommand{\br}{\mathbb{R}}

\newcommand{{\grinv}}{{\Cal G}^{-r}}

\newcommand{\ba}{\backslash}

\newcommand{\G}{\Gamma}

\newcommand{\Cal}{\mathcal}

\newcommand{\la}{\langle}
\newcommand{\ra}{\rangle}

\newcommand{\bp}{\begin{pmatrix}}
\newcommand{\ep}{\end{pmatrix}}
\renewcommand{\be}{\begin{equation}}
\newcommand{\ee}{\end{equation}}

\renewcommand{\bp}{{\rm bp}}

\newcommand{\SO}{\operatorname{SO}}

\renewcommand{\L}{\Cal L}

\newcommand{\PS}{\op{PS}}

\newcommand{\norm}[1]{\lVert #1 \rVert}

\newcommand{\op}{\operatorname}\newcommand{\supp}{\operatorname{supp}}

\newcommand{\BR}{\operatorname{BR}}

\newcommand{\ga}{\gamma}
\newcommand{\F}{\mathcal F}
\newcommand{\La}{\Lambda}
\renewcommand{\i}{\op{i}}

\def\scrB{{\mathcal B}}

\def\scrL{{\mathcal L}}

\def\e{\mathrm{e}}
\def\i{\mathrm{i}}

\def\SO{\operatorname{SO}}

\def\supp{\operatorname{supp}}

\newcommand{\GaG}{\Gamma\backslash G}

\newcommand{\pc}{P^{\circ}}

\newcommand{\fa}{\mathfrak{a}}

\newcommand{\Ga}{\Gamma}\newcommand{\bb}{\mathbb}
\newcommand{\cal}{\mathcal}
\renewcommand{\e}{\varepsilon}
\renewcommand{\epsilon}{\e}

\newcommand{\dg}{D_\Gamma^{\star}}
\newcommand{\inte}{\op{int}}

\renewcommand{\la}{\lambda}

\renewcommand{\dg}{D_{\Gamma}^{\star}}
\begin{document}

\title[Uniqueness of Conformal measures]{Uniqueness of Conformal measures  and local mixing
for Anosov groups}

\author{Sam Edwards, Minju Lee, and Hee Oh}
\address{School of Mathematics, University of Bristol, BS8 1QU, Bristol}
\address{School of mathematics, Institute for Advanced Study, Princeton, NJ 08540}
\address{Mathematics department, Yale university, New Haven, CT 06520}
\address{}

\begin{abstract} In the late seventies, Sullivan showed that
for a convex cocompact subgroup $\Ga$ of $\SO^\circ(n,1)$ with critical exponent $\delta>0$, any $\Ga$-conformal measure on $\partial \bH^n$ of dimension $\delta$ is necessarily supported on the limit set $\Lambda$ and that the conformal measure of dimension $\delta$ exists uniquely. We prove an analogue of this theorem for any Zariski dense Anosov subgroup $\Ga$ of a connected semisimple real algebraic group $G$ of rank at most $3$. We also obtain the local mixing for generalized BMS measures on $\Ga\ba G$
including Haar measures.
\end{abstract}

\email{samuel.edwards@bristol.ac.uk}
\email{minju@ias.edu}
\email{hee.oh@yale.edu}
\thanks{Edwards, Lee and Oh respectively
supported by funding from the Heilbronn Institute for Mathematical Research, 
 and the NSF grant No. DMS-1926686 (via the
Institute for Advanced Study), and the NSF grant No. DMS-1900101.
}

\maketitle

{\footnotesize
\mbox{ }\hfill Dedicated to Gopal Prasad on the occasion of his 75th birthday with respect \quad\quad\quad}\\

\section{Introduction}
 Let $(X,d)$ be a Riemannian symmetric space of rank one
and $\partial X$ the geometric boundary of $X$. Let $G
=\op{Isom}^+ X$ denote the group of orientation preserving isometries and $\G<G$ a non-elementary discrete subgroup.
Fixing $o\in X$,
a Borel probability measure $\nu$ on $\partial X$ is called a $\Ga$-conformal measure of dimension $s>0$ if
for all $\ga\in\Ga$ and $\xi\in\partial X$,
$$
\frac{d\ga_* \nu}{d\nu}(\xi) =e^{s(\beta_\xi (o, \gamma o))}
$$
 where $\beta_\xi(x,y)=\lim_{z\to\xi}d(x,z)-d(y,z)
$ denotes the Busemann function.

Let $\delta>0$ denote the critical exponent of $\Ga$, i.e., the abscissa of the convergence of the Poincare series $\sum_{\ga\in \Ga} e^{-s d(\ga o, o)}$. 
The well-known construction of
Patterson and Sullivan (\cite{Pa}, \cite{Su}) provides a $\Ga$-conformal measure of dimension $\delta$ supported on the limit set $\La$, called the Patterson-Sullivan ($\PS$) measure.
A discrete subgroup $\Ga<G$ is called \textit{convex cocompact} if  $\Ga$ acts cocompactly on some nonempty convex subset of $X$. 
\begin{thm}[Sullivan] \cite{Su}\label{sul} If $\G$ is convex cocompact, then any $\G$-conformal measure on $\partial X$ of dimension $\delta$ is necessarily supported on $\La$. Moreover, the $\PS$-measure is the unique $\Ga$-conformal measure of dimension $\delta$.
\end{thm}

In this paper, we extend this result to Anosov subgroups, which may be regarded as  higher rank analogues of convex cocompact subgroups of rank one groups.
Let $G$ be a connected semisimple real algebraic group and $P$ a minimal parabolic subgroup of $G$.
Let $\F:=G/P$ be the Furstenberg boundary, and $\cal F^{(2)}$ the unique open $G$-orbit in $\cal F\times\cal F$ under the diagonal action of $G$.
In the whole paper, we let $\G$
be a Zariski dense Anosov subgroup of $G$ with respect to $P$. This means that there exists a 
representation $\Phi: \Sigma \to G$ of a Gromov hyperbolic group $\Sigma$ with $\Ga=\Phi(\Sigma)$, which induces a continuous equivariant map $\zeta$ from the Gromov boundary $\partial \Sigma$ to $\cal F$ such that $(\zeta(x), \zeta(y))\in \F^{(2)}$ for all $x\ne  y\in \partial \Sigma$. This definition is due to
 Guichard-Wienhard \cite{GW}, generalizing that of Labourie \cite{La}.

Let $A<P$ be a maximal real split torus of $G$ and $\frak a:=\op{Lie}(A)$.
Given a linear form $\psi\in \mathfrak a^*$,
a Borel probability measure $\nu$ on $\F$ is called a $(\Gamma, \psi)$-conformal measure if, for any $\ga\in \G$ and $\xi\in \F$,
 \be\label{gc0} \frac{d\ga_* \nu}{d\nu}(\xi) =e^{\psi (\beta_\xi (e, \gamma))}\ee
where $\beta$ denotes the $\frak a$-valued Busemann function (see \eqref{Bu} for the definition).
Let $\La\subset \F$ denote the limit set of $\Ga$, which is the unique $\Gamma$-minimal subset (see \cite{Ben}, \cite{LO}). A $(\Ga ,\psi)$-conformal measure {\it{supported on}} $\La$ will be called a $(\Ga, \psi)$-$\PS$ measure. Finally,
a $\Gamma$-PS measure means a $(\Ga, \psi)$-PS measure for some $\psi\in \fa^*$.

Fix a positive Weyl chamber $\frak a^+\subset\frak a$ and let $\L_\Ga\subset \fa^+$ denote the limit cone of $\Ga$. Benoist \cite{Ben} showed that $\L_\Ga$ is a convex cone with non-empty interior, using the well-known theorem of Prasad \cite{Pr} on the existence of an $\br$-regular element in any Zariski dense subgroup of $G$.
Let $\psi_\Ga:\fa\to \br\cup \{-\infty\}$ denote the growth indicator function of $\Ga$ as defined in \eqref{def.GI}. Set
\be \label{ddgg} D_\Gamma^\star:=\{\psi\in \fa^*:\psi\ge \psi_\Ga,\;  \psi(u)= \psi_\Ga(u)\text{
for some  $u\in \mathcal L_\Ga\cap  \op{int} \fa^+ $}\}.\ee
As $\Ga$ is Anosov, for any $\psi\in \dg$,
there exist a unique unit vector $u\in \inte\L_\Ga$,
 such that
$\psi(u)=\psi_\Ga(u)$, and a unique $(\Gamma, \psi)$-$\PS$ measure $\nu_{\psi}$.
Moreover, this gives bijections among
$$\dg \simeq \{u\in \inte\L_\Ga:\|u\|=1\} \simeq \{\Ga\text{-PS measures on $\La$}\}$$
(see \cite{ELO}, \cite{LO}). When $G$ has rank one, $\dg=\{\delta\}$.
Therefore the following generalizes Sullivan's theorem \ref{sul}.
We denote the real rank of $G$ by $\op{rank}G$, i.e.,
$\op{rank}G=\text{dim}\,\fa$.
\begin{thm}\label{main} Let $\op{rank}G\leq 3$.
For any $\psi\in \dg$, any $(\Ga, \psi)$-conformal measure on $\F$
is necessarily supported on $\La$. Moreover, the
$\PS$ measure $\nu_\psi$ is
the unique $(\Ga, \psi)$-conformal measure on $\F$.
\end{thm}

Our proof of Theorem \ref{main} is obtained by combining
the rank dichotomy
theorem established by Burger, Landesberg, Lee, and Oh \cite{BLLO}
and the local mixing property of a generalized Bowen-Margulis-Sullivan measure (Theorem \ref{thm.GE}), which generalizes our earlier work \cite{ELO}. Indeed,
our proof yields that under the hypothesis of Theorem \ref{main}, any $(\Ga, \psi)$-conformal measure on $\F$
is supported on the $u$-directional radial limit set
$\La_u$ (see \eqref{lau}) where $\psi(u)=\psi_\Ga(u)$.

\medskip

We end the introduction by the following:

\noindent{\bf Open problem:}
Is Theorem \ref{main} true without the hypothesis $\text{rank } G\le 3$?

\section{Local mixing of
Generalized Bowen-Margulis-Sullivan measures}\label{s1}
Let $G$ be a connected semisimple real algebraic group and $\Gamma<G$ a Zariski dense discrete subgroup. Let $P=MAN$ be a minimal parabolic subgroup of $G$ with fixed Langlands decomposition so that $A$ is a maximal real split torus, $M$ is the centralizer of $A$ and $N$ is the unipotent radical of $P$.

In \cite[Prop. 6.8]{ELO}, we proved that local mixing of a BMS-measure on $\Ga\ba G/M$ implies local mixing of the Haar measure on $\Gamma\ba G/M$. In this section, we provide a generalized version of this statement, where we replace the Haar measure by any generalized BMS-measure and also work on the space $\Gamma\ba G$, rather than on $\Gamma\ba G/M$. We refer to \cite{ELO} for a more detailed description of a generalized BMS-measure, while
only briefly recalling its definition here.

Let $\fa=\op{Lie}(A)$ and fix a positive Weyl chamber $\fa^+<\fa$ so that $\log N$ consists of positive root subspaces. We also fix a maximal compact subgroup $K<G$ so that the Cartan decomposition $G=K(\exp \fa^+) K$ holds. Denote by $\mu:G\to \fa^+$ the Cartan projection, i.e., for $g\in G$, $\mu(g)\in \fa^+$ is the unique element such that $g\in K\exp \mu(g)K$. Denote by $\L_\Ga\subset \fa^+$ the limit cone of $\Gamma$, which is the asymptotic cone of $\mu(\Ga)$, i.e.,
$\L_\Ga=\{\lim t_i \mu(\ga_i)\in \fa^+: t_i\to 0, \ga_i\in \Ga\}$.  The Furstenberg boundary
$\F=G/P$ is isomorphic to $K/M$ as $K$ acts on $\F$ transitively with $K\cap P=M$.

The $\frak a$-valued Busemann function $\beta: \cal F\times G \times G \to\frak a $ is defined as follows: for $\xi\in \cal F$ and $g, h\in G$,
 \be\label{Bu} \beta_\xi ( g, h):=\sigma (g^{-1}, \xi)-\sigma(h^{-1}, \xi)\ee 
where the Iwasawa cocycle $\sigma(g^{-1},\xi)\in \fa$
is defined by the relation $g^{-1}k \in K \exp (\sigma(g^{-1}, \xi)) N$ for $\xi=kP$, $k\in K$.

 The growth indicator function $\psi_{\Gamma}\,:\,\frak a^+ \rightarrow \br \cup\lbrace- \infty\rbrace$  is defined as a homogeneous function, i.e., $\psi_\Gamma (tu)=t\psi_\Gamma (u)$ for all $t>0$, such that
  for any unit vector $u\in \frak a^+$,
 \begin{equation} \label{def.GI}
\psi_{\Gamma}(u):=\inf_{{u\in\cal C},{\mathrm{open\;cones\;}\cal C\subset \fa^+}}\tau_{\cal C}
\end{equation}
where $\tau_{\cal C}$ is the abscissa of convergence of $\sum_{\ga\in\Ga, \mu(\ga)\in\cal C}e^{-t\norm{\mu(\ga)}}$ and
the norm $\|\cdot\|$  on $\frak a$
is the one induced from the Killing form on $\frak g$.

 Denote by $w_0\in K$ a representative of the unique element of the Weyl group $N_K(A)/M$ such that $\op{Ad}_{w_0}\mathfrak a^+= -\mathfrak a^+$.
  The opposition involution  $\i:\mathfrak a \to \mathfrak a$ is defined by
  $$\i (u)= -\op{Ad}_{w_0} (u).$$ Note that $\i$ preserves $\inte \L_\Ga$.

\subsection*{The generalized BMS-measures $m_{\nu_1,\nu_2}$.}

For $g\in G$, we consider the following visual images:
$$g^+=gP\in \F \quad \text{ and}\quad g^-=gw_0P\in \F.$$ Then the map $$gM\mapsto (g^+, g^-, b=\beta_{g^-}(e, g))$$ gives a homeomorphism
 $G/M\simeq  \F^{(2)}\times \fa $, called the Hopf parametrization of $G/M$.

For a pair of linear forms $\psi_1, \psi_2\in \mathfrak a^*$ and a pair of  $(\Gamma,\psi_1)$ and $(\Gamma,\psi_2)$ conformal measures $\nu_1$ and $\nu_2$ respectively, define a locally finite Borel measure $\tilde {m}_{\nu_1,\nu_2}$ on $G/M$ 
  as follows: for $g=(g^+, g^-, b)\in \F^{(2)}\times \mathfrak a$,
\begin{equation}\label{eq.BMS0}
d\tilde m_{\nu_1, \nu_2} (g)=e^{\psi_1 (\beta_{g^+}(e, g))+\psi_2( \beta_{g^-} (e, g )) } \;  d\nu_{1} (g^+) d\nu_{2}(g^-) db,
\end{equation}
  where $db=d\ell (b) $ is the Lebesgue measure on $\mathfrak a$.
  By abuse of notation, we also denote by $\tilde m_{\nu_1, \nu_2}$ the $M$-invariant measure on $G$ induced by $\tilde m_{\nu_1, \nu_2}$.
This is always left $\Ga$-invariant  and we denote by
$ m_{\nu_1, \nu_2}$ the $M$-invariant measure on $\Ga\ba G$ induced by $\tilde m_{\nu_1, \nu_2}$.

\subsection*{The generalized BMS$^*$-measures $m^*_{\nu_1,\nu_2}$.}
Similarly, with a different Hopf parametrization 
$$gM\mapsto (g^+,g^-, b=\beta_{g^+}(e,g))$$ (that is, $g^-$ replaced by $g^+$ in the subscript for $\beta$), we define the following measure 
\begin{equation}\label{eq.BMS*}
d\tilde m_{\nu_1, \nu_2}^* (g)=e^{\psi_1 (\beta_{g^+}(e, g))+\psi_2( \beta_{g^-} (e, g )) } \;  d\nu_{1} (g^+) d\nu_{2}(g^-) db
\end{equation} first on $G/M$ and then the $M$-invariant measure $d m_{\nu_1, \nu_2}^*$ on $\Gamma\ba G$.
One can check 
\be\label{sw} m_{\nu_1,\nu_2}^{*}=m_{\nu_2,\nu_1}.w_0 .\ee 

\begin{lem}\label{dudu}
If $\psi_2=\psi_1\circ \i$, then $m_{\nu_1, \nu_2}=m^*_{\nu_1, \nu_2}$. 
\end{lem}
\begin{proof} When $\psi_2=\psi_1\circ\i$, we can check that $m_{\nu_2,\nu_1}.w_0=m_{\nu_1,\nu_2}$, which implies the claim by \eqref{sw}.
\end{proof}

\subsection*{PS-measures on $gN^{\pm}$} \label{psm} Let $N^-=N$ and $N^+=w_0 Nw_0^{-1}$.
To  a given $(\Ga,\psi)$-conformal measure $\nu$ and $g\in G$,
we define the following associated measures on $gN^{\pm}$: 
for $n\in N^+$ and $h\in N^-$, 
\begin{align*}\label{eq.PSN}   
d\mu_{gN^+,\nu}{}(n)&:=e^{\psi(\beta_{(gn)^+} (e,gn) ) }d\nu ((gn)^+),\text{ and }\\
d\mu_{gN^-,\nu}{}(h)&:=e^{\psi (\beta_{(gh)^-}(e, gh)) }d\nu ((gh)^-).
\end{align*}
Note that these are left $\Gamma$-invariant; for any  $\gamma\in \Gamma$ and $g\in G$, 
$\mu_{\gamma gN^\pm,\nu }{}= \mu_{gN^{\pm},\nu}{}.$
For a given Borel subset $X\subset\Ga\ba G$, define the measure $\mu_{gN^+,\nu}{}|_X$ on $N^+$ by
$$
d\mu_{gN^+,\nu}{}|_X
(n)=\mathbbm{1}_X
([g]n)\,d\mu_{gN^+,\nu}{}(n);
$$
note that here the notation $|_X
$ is purely symbolic, as $\mu_{gN^+,\nu}{}|_X
$ is not a measure on $X$.
Set $P^\pm:=MAN^\pm$.
For $\e>0$ and $\star=N, N^+, A, M$,
let $\star_\e$ denote the $\e$-neighborhood of $e$ in $\star$. We then set $P^{\pm}_\e=N^{\pm}_\e A_\e M_\e$.

We recall the following lemmas from \cite{ELO}:
\begin{lem}\label{thickcor}
\cite[Lem. 5.6, Cor. 5.7]{ELO} We have:
\begin{enumerate}
\item
 For any fixed $\rho\in C_c(N^\pm)$ and $g\in G$,
the map $N^\mp\to \br$ given by $n \mapsto \mu_{gn N^{\pm},\nu}{}(\rho)$ is continuous.
\item
Given $\epsilon>0$ and $g\in G$, there exist $R>1$ and a non-negative $\rho_{g,\epsilon}\in C_c(N_R)$ such that $\mu_{gnN,\nu}{}(\rho_{g,\epsilon})>0$ for all $ n\in N_{\epsilon}^+$. 
\end{enumerate}
\end{lem} 
\begin{lem}\label{MANconj}\cite[Lem. 4.2]{ELO} For any $g\in G,\,a\in A,\,n_0,n\in N^+$, we have
\begin{equation*}
d(\theta^{-1}_*\mu_{gN^+,\nu})(n)
=e^{-\psi(\log a)}d\mu_{gan_0N^+,\nu}{}(n),
\end{equation*}
 where $\theta:N^+\to N^+$ is given by $\theta(n)=an_0na^{-1}.$
\end{lem}
\begin{lemma}\label{lem.DC}\cite[Lem. 4.4 and 4.5]{ELO}
For $i=1,2$, let $\psi_i\in\frak a^*$ and $\nu_i$ a $(\Ga,\psi_i)$-conformal measure.
Then
\begin{enumerate}
\item
For $g\in G$, $f\in C_c(gN^+P)$, and $nham\in N^+NAM$,
\begin{multline*}
\tilde{m}_{\nu_1, \nu_2}
(f)=\\ \int_{N^+}\left(\int_{NAM}  f(gnham)e^{(\psi_1-\psi_2\circ\i)(\log a)}\,dm\, da \; d\mu_{gnN,\nu_2}{}(h)\right)d\mu_{gN^+,\nu_1}{}(n).
\end{multline*}
\item
For $g\in G$, $f\in C_c(gPN^+)$, and $hamn\in NAMN^+$,
\begin{multline*}
\tilde{ m}_{\nu_1, \nu_2}^*
(f)=\\ \int_{NAM}\left(\int_{N^+}f(ghamn)\,d\mu_{ghamN^+,\nu_1}{}(n)\right)e^{-\psi_2\circ\i(\log a)}\,dm\,da\,d\mu_{gN,\nu_2}{}(h).
\end{multline*}

\end{enumerate}
\end{lemma}

\medskip 
\subsection*{Local mixing}
Let $P^\circ$ denote the identity component of $P$ and $\frak Y_\Ga$ denote the set of all $P^\circ$-minimal subsets of $\Ga\ba G$.
While there exists a unique $P$-minimal subset of $\Ga\ba G$
given by $\{[g]\in \Ga\ba G: g^+\in \La\}$,
there may be more than one $\pc$-minimal subset. Note that
$\#\frak Y_\Ga \le [P: \pc]=[M:M^\circ]$. 
Set $\Omega=\{ [g]\in \Ga\ba G: g^{\pm}\in \La\}$ and write 
$$\frak Z_\Ga=\{Y\cap \Omega\subset \Ga\ba G: Y\in \frak Y_\Ga\}.$$
Note that for each $Y\in \frak Y_\Ga$, we have
$Y=(Y\cap \Omega)N$ and the collection $\{(Y\cap \Omega)N^+:Y\in \frak Y_\Ga\}$ is in one-to-one correspondence with the set
of $(M^\circ AN^+)$-minimal subsets of $\Ga\ba G$.

In the rest of the section, we fix a unit vector $u\in \L_\Ga\cap \inte \fa^+$, and set
$$a_t=\exp (tu)\quad\text{ for $t\in \br$.}$$
We also fix
$$\psi_1\in \fa^* \quad\text{ and}\quad \psi_2:=\psi_1\circ \i\in \fa^*.$$ 
 For each $i=1,2$,
we  fix  a $(\Ga, \psi_i)$-$\PS$ measure $\nu_i$ on $\F$. 
We will assume that the associated BMS-measure
$\mathsf m={m}_{\nu_1, \nu_2}$ satisfies the local mixing property for the $\{a_t: t\in \br \}$-action in the following sense:

\noindent{\bf Hypothesis on $\mathsf m=m_{\nu_1, \nu_2}$:}
there exists a proper continuous function $\Psi : \bb (0,\infty)\to\bb (0,\infty)$
 such that for all $f_1,\,f_2\in C_c(\Ga\ba G)$,
\begin{equation}\label{eq.hypo}
\lim_{t\rightarrow+\infty} \Psi(t)\int_{\GaG} f_1 (x a_t  ) f_2(x)  \,d\mathsf m (x)=\sum_{Z\in\frak Z_\Ga}
   \mathsf m|_Z(f_1)\,\mathsf m|_Z(f_2).
\end{equation}

The main goal in this section is to obtain the following local mixing property for a generalized BMS-measure $m_{\lambda_1,\lambda_2}$ from that of $\mathsf m$
(note that $\lambda_1$ and $\lambda_2$ are not assumed to be supported on $\La$):
\begin{Thm}\label{prop.mixH000} For $i=1,2$, let $\varphi_i\in \fa^*$ and $\lambda_i$ be 
 a $(\Gamma, \varphi_i)$-conformal measure on $\F$.
Then for all $f_1, f_2\in C_c(\Ga\ba G)$, we have
\begin{multline*}
\lim_{t\rightarrow+\infty} \Psi(t)e^{(\varphi_1-\psi_1)(t  u )} \int_{\GaG} f_1 (x a_t)f_2(x)\,dm^*_{\lambda_1,\lambda_2}{}(x) \\
= \sum_{Z\in\frak Z_\Ga}
m_{\lambda_1,\nu_2}|_{ZN^+}(f_1)\, m_{\nu_1,\lambda_2}^{*}|_{ZN}(f_2).
\end{multline*}
\end{Thm}

\begin{Rmk} \rm If $\varphi_2=\varphi_1\circ\i$,
we may replace $m_{\la_1,\la_2}^*$ by
$m_{\la_1,\la_2}$ in Theorem \ref{prop.mixH000}
by Lemma \ref{dudu}.
For general $\varphi_1,\varphi_2$, we get,
using the identity \eqref{sw}: for all $f_1, f_2\in C_c(\Ga\ba G)$, we have
\begin{multline*}
\lim_{t\rightarrow+\infty} \Psi(t)e^{(\varphi_1-\psi_1)(t  u )} \int_{\GaG} f_1 (x a_{-t})f_2(x)\,dm_{\lambda_2,\lambda_1}{}(x) \\
= \sum_{Z\in\frak Z_\Ga}
m_{\nu_2,\lambda_1}^*|_{ZN^+}(f_1)\, m_{\lambda_2, \nu_1}|_{ZN}(f_2).
\end{multline*}
\end{Rmk}

In order to prove Theorem \ref{prop.mixH000}, we first deduce equidistribution of translates of $\mu_{gN^+,\nu_1}{}$ from the local mixing property of ${\mathsf m}$ (Proposition \ref{p1}), and then convert this into equidistribution of translates of $\mu_{gN^+,\lambda_1}{}$ (Proposition \ref{p2}).
\begin{prop}\label{p1}
 For any $x=[g] \in \GaG$, $f\in C_c(\GaG)$, and $\phi\in C_c(N^+)$,
\begin{equation} \label{PSequi} 
\lim_{t\rightarrow +\infty} \Psi(t)\int_{N^+} f\ (x na_{t} )\phi(n)\,d\mu_{gN^+, \nu_1}{}(n)= \sum_{Z\in\frak Z_\Ga} {\mathsf m}|_Z(f)\,\mu_{gN^+,\nu_1}{}|_{ZN}(\phi).
\end{equation}
\end{prop}
\begin{proof}
Let $x=[g]$, and $\epsilon_0>0$ be such that $\phi\in C_c(N_{\epsilon_0}^+)$. For simplicity of notation,
we write $d\mu_{\nu_1}=d\mu_{gN^+, \nu_1}$ throughout the proof.
By Lemma \ref{thickcor}, we can choose $R>0$ and a nonnegative $\rho_{g,\epsilon_0}\in C_c(N_R)$ such that
\begin{equation*}
\mu_{gnN,\nu_2}{}(\rho_{g,\epsilon_0})>0\quad \text{ for all  }n\in N_{\epsilon_0}^+.
\end{equation*} 
Given any $\epsilon>0$, choose a non-negative function  $q_{\epsilon}\in C_c(A_{\epsilon}M_{\e})$ satisfying $\int_{AM}q_{\epsilon}(am)\,da\,dm=1$.
Then 
\begin{align}\label{fthickening}
&\int_{N^+} f(x n a_{t})\phi(n)\,d\mu_{\nu_1}{}(n)=
\\\notag&\int_{N^+} f(xn a_{t} )\phi(n)\left( \tfrac{1}{\mu_{gnN,\nu_2}{}(\rho_{g,\epsilon_0})}\int_{NA} \rho_{g,\epsilon_0}(h)q_{\epsilon}(am)\,da\,dm\,d\mu_{gnN,\nu_2}{}(h)\!\right)d\mu_{\nu_1}{}(n)
\\\notag&=\int_{N^+} \left(\int_{NA}f(xna_{t} ) \tfrac{\phi(n)\rho_{g,\epsilon_0}(h)q_{\epsilon}(am)}{\mu_{gnN,\nu_2}{}(\rho_{g,\epsilon_0})} \,da\,dm\,d\mu_{gnN,\nu_2}{}(h)\!\right)d\mu_{\nu_1}{}(n).
\end{align}
We now define $\tilde \Phi_{\epsilon}\in C_c(gN_{\epsilon_0}^+N_R A_{\epsilon}M_{\e})\subset C_c(G)$ and $\Phi_{\e}\in C_c(\Ga\ba G)$ by
\begin{equation*}
\tilde \Phi_{\epsilon}(g_0):=\begin{cases}\frac{\phi(n)\rho_{g,\epsilon_0}(h)q_{\epsilon}(am)}{\mu_{gnN,\nu_2}{}(\rho_{g,\epsilon_0})}\qquad&\mathrm{if\;}g_0=gnham,\\0\qquad&\mathrm{otherwise,}
\end{cases}
\end{equation*}
and ${\Phi}_{\epsilon}([g_0]):=\sum_{\gamma\in \Gamma} \tilde\Phi_{\epsilon}(\gamma g_0)$.
Note that the continuity of $\tilde  \Phi_{\epsilon}$ follows from Lemma \ref{thickcor}.
We now assume without loss of generality that $f\geq 0$ and define, for all $\epsilon>0$, functions $f_{\epsilon}^{\pm}$ as follows:  for all $ z\in \GaG$,
$$
f_{\epsilon}^+(z):= \sup_{b\in N_{\epsilon}^+P_{\e}} f(zb)\;\;\text{and} \;\;
f_{\epsilon}^-(z):= \inf_{b\in N_{\epsilon}^+P_{\e}} f(zb).$$

Since $u\in\inte\fa^+$, for every $\epsilon>0$, there exists $t_0(R,\epsilon)> 0$ such that
\begin{equation*}
a_t^{-1}  
N_R
a_t
 \subset N_{\epsilon}\qquad \text{ for all }t\geq 
 { t_0(R,\epsilon)}. 
\end{equation*}
Then, as $\mathrm{supp}(\tilde  \Phi_{\epsilon})\subset gN_{\epsilon_0}^+ N_{R} A_{\epsilon}M_{\e}$, we have
\begin{align}\label{fplus}
f(xn{a_{t}} )\tilde  \Phi_\epsilon(gnham) \leq f_{3\epsilon}^+ (x nh a ma_{t} )\tilde  \Phi_\epsilon(gnham)
\end{align}
for all $nh am\in N^+NAM$
and $t\geq t_0(R,\epsilon)$.
We now use $f_{3\epsilon}^+$ to give an upper bound on the limit we are interested in; $f_{3\epsilon}^-$ is used in an analogous way to provide a lower bound. Entering the definition of $\Phi_{\epsilon}$ and the above inequality \eqref{fplus} into \eqref{fthickening} gives
\begin{multline*}
\limsup_{t\rightarrow+\infty}\, \Psi(t)\int_{N^+} f(xn a_{t} )\phi(n)\,d\mu_{\nu_1}{}(n)
\\ \leq \limsup_{t\rightarrow+\infty} \,\Psi(t) \\ \int_{N^+} \int_{NAM}f_{3\epsilon}^+(x nham a_{t} ) \tilde \Phi_{\epsilon}(g nham)dm\,da\,d\mu_{gnN,\nu_2}{}(h)\,d\mu_{\nu_1}{}(n)
\\
\leq \limsup_{t\rightarrow+\infty} \,\Psi(t)e^{\e\norm{\psi_1-\psi_2\circ\i}}
\int_{N^+} \int_{NAM}f_{3\epsilon}^+(x nham a_{t} ) \tilde \Phi_{\epsilon}(g nham)\\ e^{(\psi_1-\psi_2\circ\i)(\log a)}\,dm\,da\,d\mu_{gnN,\nu_2}{}(h) \,d\mu_{\nu_1}{}(n)
\\=\limsup_{t\rightarrow+\infty}\, \Psi(t)e^{\e\norm{\psi_1-\psi_2\circ\i}}\int_{G}f_{3\epsilon}^+([g_0] a_t ) \tilde  \Phi_{\epsilon}(g_0)\,d\tilde{\mathsf m}(g_0)
\\=\limsup_{t\rightarrow+\infty}\, \Psi(t)e^{\e\norm{\psi_1-\psi_2\circ\i}}\int_{\GaG} f_{3\epsilon}^+([g_0] a_t ) {\Phi}_{\epsilon}([g_0])\,d{\mathsf m}([g_0]),
\end{multline*}
where
$\norm{\cdot}$ is the operator norm on $\frak a^*$ and
Lemma \ref{lem.DC} was used in the second to last line of the above calculation.
By the standing assumption \eqref{eq.hypo}, we have
\begin{align*}
\limsup_{t\rightarrow+\infty}\, &\Psi(t)\int_{N} f(x n a_{t} )\phi(n)\,d\mu_{gN,\nu_2}{}(n) 
\\&\leq   e^{\e\norm{\psi_1-\psi_2\circ\i}}\sum_{Z\in\frak Z_\Ga} \,{\mathsf m}|_Z(f_{3\epsilon}^+){\mathsf m}|_Z({\Phi}_{\epsilon})\\
&= e^{\e\norm{\psi_1-\psi_2\circ\i}}\sum_{Z\in\frak Z_\Ga} \,{\mathsf m}|_Z(f_{3\epsilon}^+)\tilde{\mathsf m}|_{\tilde Z}( \tilde \Phi_{\epsilon}),
\end{align*}
where $\tilde Z\subset G$ is a $\Ga$-invariant lift of $Z$.
Using Lemma \ref{lem.DC}, for all $0<\e\ll1$,
\begin{align*}
&\tilde{\mathsf m}|_{\tilde { Z}}( \tilde  \Phi_{\epsilon})\\
&=\int_{N^+} \left(\int_{NAM}\tilde  \Phi_{\epsilon}\mathbbm{1}_{\tilde Z}(g nham)e^{(\psi_1-\psi_2\circ\i)(\log a)}\,da\,dm\,d\mu_{gnN,\nu_2}{}(h)\right)\,d\mu_{\nu_1}{}(n)\le 
\\&e^{\e\norm{\psi_1-\psi_2\circ\i}}\int_{N^+} \tfrac{\phi(n)\mathbbm{1}_{
\tilde Z N}(gn)}{\mu_{gnN,\nu_2}{}(\rho_{g,\epsilon_0})}\left(\int_{NAM} \rho_{g,\epsilon_0}(h)q_{\epsilon}(am)\,da\,dm\,d\mu_{gnN,\nu_2}{}(h)\!\right)\,d\mu_{\nu_1}{}(n)
\\&\leq e^{\e\norm{\psi_1-\psi_2\circ\i}} \mu_{\nu_1}{}|_{ZN}(\phi),
\end{align*}
where we have used the facts that $\tilde Z$ is invariant under the right translation of identity component $M^\circ$ of $M$,
and $\supp\nu_2=\La$ as well as the identity $\mathbbm{1}_{\tilde Z}(gnha)=\mathbbm{1}_{\tilde ZN}(gn)\mathbbm{1}_{\La}(gnh^+)$ 
(we remark that $\supp\nu_2=\La$ is not necessary for the upper bound as $\mathbbm{1}_{\tilde Z}(gnha)\leq \mathbbm{1}_{\tilde ZN}(gn)$, but  needed for the lower bound).
Since $\epsilon>0$ was arbitrary, taking $\epsilon\rightarrow0$ gives
\begin{equation*}
\limsup_{t\rightarrow+\infty}\, \Psi(t)\int_{N^+} f(x n a_{t} )\phi(n)\,d\mu_{\nu_1}{}(n) \leq 
\sum_{Z\in\frak Z_\Ga} {\mathsf m}|_Z(f)\,\mu_{\nu_1}{}|_{ZN}(\phi).
\end{equation*}
The lower bound given by replacing $f^+_{3\epsilon}$ with $f^-_{3\epsilon}$ in the above calculations completes the proof.
\end{proof}
\begin{prop}\label{p2} 
For any $x=[g] \in \GaG$, $f\in C_c(\Ga\ba G)$ and $\phi\in C_c(N^+)$,
\begin{multline*}
\lim_{t\rightarrow+\infty} \Psi(t)
e^{(\varphi_1-\psi_1)(t  u )}
\int_{N^+}f\big(x n a_{t}\big)\phi(n)\,d\mu_{gN^+,\lambda_1}{}(n)\\ =\sum_{Z\in\frak Z_\Ga}   
m_{\lambda_1,\nu_2}{}|_{ZN^+}(f)\, \mu_{ gN^+,\nu_1}{}|_{ZN}(\phi).
\end{multline*}
\end{prop}

\begin{proof} For $\epsilon_0>0$, set  $\scrB_{\epsilon_0}=P_{\e_0}N_{\epsilon_0}^+$.
Given $x_0\in \GaG$, let $\epsilon_0(x_0) $ denote the maximum number $r$ such that the map $G\rightarrow\GaG$ given by $h\mapsto x_0 h$ for $h\in G$ is injective on $\scrB_{r}$.
By using a partition of unity if necessary, it suffices to prove that for any $x_0\in \GaG$ and $\e_0=\e_0(x_0)$, the claims of the proposition hold for any non-negative $f\in C(x_0 \scrB_{\epsilon_0})$, non-negative $\phi\in C(N_{\epsilon_0}^+)$, and $x=[g]\in x_0 \scrB_{\epsilon_0}$.
Moreover, we may assume that $f$ is given as
\begin{equation*}
f([g])=\sum_{\gamma \in \Gamma} \tilde{f} (\gamma g)\qquad \text{ for all }  g\in G,
\end{equation*}
for some non-negative $\tilde{f}\in C_c(g_0\scrB_{\epsilon_0}) $.
For simplicity of notation, we write
$\mu_{\lambda_1}=\mu_{gN^+, \lambda_1}$. Note that for $x=[g] \in  [g_0] \scrB_{\epsilon_0}$,
\begin{align}\label{eq.un1}
\int_{N^+} f([g]n a_{t})\phi(n)\,d\mu_{\lambda_1}{}(n)=\sum_{\gamma \in \Gamma} \int_{N^+} \tilde{f}(\gamma g n a_{t})\phi(n)\,d\mu_{\lambda_1}{}(n).
\end{align}
Note that $\tilde{f}(\gamma g n a_{t})=0$ unless $\gamma g n a_{t}
\in g_0 \scrB_{\epsilon_0}$.
Together with the fact that $\mathrm{supp}(\phi)\subset 
N_{\epsilon_0}^+
$, it follows that the summands in \eqref{eq.un1} are non-zero only for finitely many elements $\gamma \in \Gamma \cap  g_0 \scrB_{\epsilon_0}
a_{-t}  N_{\epsilon_0}^+ g^{-1}$.

Suppose $\gamma g N_{\epsilon_0}^+ a_{t}
\cap g_0 \scrB_{\epsilon_0}\neq\emptyset$. 
Then $\gamma g a_{t}\in g_0 P_{\epsilon_0}N^+$, and there are unique elements
$p_{t,\gamma}\in P_{\epsilon_0}$ and $n_{t,\gamma}\in N^+$ such that
\begin{equation*}
\gamma g a_{t} = g_0 p_{t,\gamma}n_{t,\gamma} \in g_0 P_{\epsilon_0} N^+.
\end{equation*}
Let ${\Gamma_{t}}$ denote the subset $\Gamma \cap g_0 (P_{\epsilon_0}N^+ )a_t^{-1}
 g^{-1}$.
Note that although $
\Gamma_{t}$ may possibly be infinite, only finitely many of the terms in the sums we consider will be non-zero.
This together with Lemma \ref{MANconj} gives
\begin{align*}
&\int_{N^+} f([g]n a_{t})\phi(n)\,d\mu_{\lambda_1}{}(n)=\sum_{\gamma \in \Gamma} \int_{N^+} \tilde{f}(\gamma g n a_{t})\phi(n)\,d\mu_{\lambda_1}{}(n)\\
&=\sum_{\gamma \in\Gamma_{t} } \int_{N^+} \tilde{f}(\ga g a_t(a_t^{-1} n a_t)\big)\phi(n)\,d\mu_{\lambda_1}{}(n)\\
&=e^{-\varphi_1(\log a_t)} \sum_{\gamma \in 
\Gamma_{t}} \int_{N^+} \tilde{f}(\gamma g a_t n)\phi(a_t n a_t^{-1})\,d\mu_{ga_tN^+,\lambda_1}{}(n) \\ 
&= e^{-\varphi_1(\log a_t)}
\sum_{\gamma \in  \Gamma_{t}} \int_{N^+} \tilde{f}\big(g_0p_{t,\gamma}n_{t,\gamma} n\big)\phi(a_t n a_t^{-1})\,d\mu_{ga_tN^+,\lambda_1}{}(n) \\
&={e^{-\varphi_1(\log a_t)}} \sum_{\gamma \in  {\Gamma_{t} }
} \int_{N^+} \tilde{f}\big(g_0p_{t,\gamma} n\big)\phi\big( a_t \,n_{t,\gamma}^{-1}n\; a_t^{-1}\big)\,d\mu_{g_0p_{t,\gamma}N^+,\lambda_1}{}(n).\end{align*}
Since $\supp(\tilde{f})\subset g_0 \scrB_{\epsilon_0}$, we have 
\begin{align*}
 \sum_{\gamma \in {\Gamma_{t}}}& \int_{N^+} \tilde{f}\big(g_0p_{t,\gamma} n\big)\phi\big( a_t \,n_{t,\gamma}^{-1}n\; a_t^{-1}\big)\,d\mu_{g_0p_{t,\gamma}N^+,\lambda_1}{}(n)
\\\leq&\sum_{\gamma \in  \Gamma_{t} } \left(\sup_{n\in N_{\epsilon_0}^+} \phi\big(
a_t \,n_{t,\gamma}^{-1}\;a_t^{-1}(a_t na_t^{-1}) \big) \right)\cdot \int_{N^+} \tilde{f}\big(g_0p_{t,\gamma} n\big)\,d\mu_{g_0p_{t,\gamma}N^+,\lambda_1}{}(n).
\end{align*}
Since $u$ belongs to $\op{int}\scrL_{\Gamma}$, there exist $t_0>0$ and $\alpha>0$
such that
\begin{equation*}
a_t N_r^+ a_t^{-1}\subset N_{r e^{-\alpha t}}^+\qquad \text{ for all }\, r>0\text{ and }t>t_0.
\end{equation*}
Therefore, for all $n\in N_{\epsilon_0}^+$
and $t>t_0$,
we have
\begin{align}\label{eq.x1}
  \phi\big(a_t\,n_{t,\gamma}^{-1}a_t^{-1}(a_t na_t^{-1}) \big)\leq  \phi_{\epsilon_0 e^{-\alpha t}}^+\big(a_t\,n_{t,\gamma}^{-1}\;a_t^{-1} \big),
\end{align}
where
\begin{equation*}
\phi^+_{\epsilon}(n):=\sup_{b\in N_{\epsilon}^+} \phi(nb)\qquad\text{ for all } n\in N^+,\,\epsilon>0.
\end{equation*}
We now have the following inequality for $t>t_0$:
\begin{align}\label{intchain}
&e^{\varphi_1(\log a_t)}
 \,\int_{N^+} f([g] n a_{t})\phi(n)\,d\mu_{\lambda_1}{}(n)
\notag \\
 &\leq\sum_{\gamma \in \Gamma_{t}} \phi_{\epsilon_0 e^{-\alpha t}}^+\big(a_t\,n_{t,\gamma}^{-1}\;a_t^{-1}
\big)\int_{N_{\epsilon_0}^+} \tilde{f}\big(g_0p_{t,\gamma} n\big)\,d\mu_{g_0p_{t,\gamma}N^+,\lambda_1}{}(n).
\end{align}
By Lemma \ref{thickcor}, we can now choose $R>0$ and $\rho\in C_c(N_R^+)$ such that $\rho(n)\geq 0$ for all $n\in N^+$, and $\mu_{ g_0pN^+,\nu_1}{}(\rho)>0$ for all $p\in P_{\epsilon_0}$.
Define $\tilde{F}\in C_c(g_0 P_{\epsilon_0}N_R^+)$ by
\begin{align*}
&\tilde{F}(g)=
\begin{cases} \tfrac{\rho(n)}{\mu_{ g_0pN^+,\nu_1}(\rho)}\int_{N_{\epsilon_0}^+} \tilde{f}\big(g_0p v\big)\,d\mu_{g_0pN^+,\lambda_1}
(v)&\text{if }g=g_0pn\in g_0 P_{\epsilon_0}N_R^+
\\
0&
\text{if }g\not\in g_0 P_{\epsilon_0}N_R^+.
\end{cases}
\end{align*}

We claim that for all $p\in P_{\epsilon_0}$ and $Z\in\frak Z_\Ga$ such that $g_0p^-\in\La$,
\begin{multline}\label{eq.1}
\int_{N^+} \tilde{F} (g_0pn)\,
d\mu_{g_0pN^+,\nu_1}|_{Z}
(n)=\int_{N_R^+} \tilde{F} (g_0pn)\,
d\mu_{g_0pN^+,\nu_1}|_Z
(n)\\
=\int_{N_{\epsilon_0}^+}(\tilde{f}\mathbbm{1}_{ZN^+})(g_0 pn)\,
d\mu_{g_0pN^+,\lambda_1}
(n).
\end{multline}
Indeed, by the assumption $\supp\nu_1=\La$ and the fact $\Omega\cap ZN^+=Z$, we have the identity $\mathbbm{1}_Z(g_0pn)\,d\mu_{g_0pN^+,\nu_1}(n)=\mathbbm{1}_{ZN^+}(g_0p)\,d\mu_{g_0pN^+,\nu_1}(n)$ and hence
\begin{align*}
   &  \int_{N^+} \tilde{F} (g_0pn)\,
d\mu_{g_0pN^+,\nu_1}|_{Z}
(n) \\
&=\int_{N^+}  \tilde{F} (g_0pn)  \mathbbm{1}_Z(g_0pn) \,
d\mu_{g_0pN^+,\nu_1}
(n)
\\
&=\int_{N^+}\tfrac{\rho(n)\mathbbm{1}_{ZN^+}(g_0p)}{\mu_{ g_0pN^+,\nu_1}(\rho)}\left(\int_{N_{\epsilon_0}^+} \tilde{f}\big(g_0p v\big)\,d\mu_{g_0pN^+,\lambda_1}
(v)\right)d\mu_{g_0pN^+,\nu_1}(n)
\\
&=\int_{N^+}\tfrac{\rho(n)}{\mu_{ g_0pN^+,\nu_1}(\rho)}\left(\int_{N_{\epsilon_0}^+} (\tilde{f}\mathbbm{1}_{ZN^+})\big(g_0p v\big)\,d\mu_{g_0pN^+,\lambda_1}
(v)\right)d\mu_{g_0pN^+,\nu_1}(n)
\\
&=\int_{N_{\epsilon_0}^+} (\tilde{f}\mathbbm{1}_{ZN^+})\big(g_0p v\big)\,d\mu_{g_0pN^+,\lambda_1}
(v).
\end{align*}

Summing up \eqref{eq.1} for all $Z\in\frak Z_\Ga$ and using $\supp\nu_1=\La$, we get
\begin{align*}
&\int_{N^+} \tilde{F} (g_0pn)\,
d\mu_{g_0pN^+,\nu_1}
(n)\\&=\sum_{Z\in\frak Z_\Ga}\int_{N^+} \tilde{F} (g_0pn)\,
d\mu_{g_0pN^+,\nu_1}|_{Z}
(n)\\
&=\sum_{Z\in\frak Z_\Ga}\int_{N_{\epsilon_0}^+}(\tilde{f}\mathbbm{1}_{ZN^+})(g_0 pn)\,
d\mu_{g_0pN^+,\lambda_1}
(n).
\end{align*}
Hence we can write
\begin{align*}
&    \int_{N_{\epsilon_0}^+}\tilde{f}(g_0 pn)\,
d\mu_{g_0pN^+,\lambda_1}
(n)\\
&=\int_{N^+} \tilde{F} (g_0pn)\,
d\mu_{g_0pN^+,\nu_1}
(n)+\int_{N_{\e_0}^+}\tilde h(g_0pn)d\mu_{g_0pN^+,\lambda_1}
(n)
\end{align*}
for some $\tilde h$ that vanishes on $\bigcup_{Z\in\frak Z_\Ga} ZN^+$.
Returning to \eqref{intchain}, we now give an upper bound.
We observe: \begin{align*}
& e^{\varphi_1(\log a_t)} \int_{N^+} f([g] n a_{t})\phi(n)\,d\mu_{\lambda_1}{}(n)\\
\leq&\,  
\sum_{\gamma \in \Gamma_{t}
} \phi_{\epsilon_0 e^{-\alpha t}}^+\big(a_t
\,n_{t,\gamma}^{-1}\;a_t^{-1}
\big)\int_{N_{\epsilon_0}^+} \tilde{f}\big(g_0p_{t,\gamma} n\big)\,d\mu_{\lambda_1}{}(n)\\
=&\,\sum_{\gamma \in \Gamma_{t}} \phi_{\epsilon_0 e^{-\alpha t}}^+\big(a_t\,n_{t,\gamma}^{-1}\;a_t^{-1}\big)
\int_{N_R^+} (\tilde{F}+\tilde h) (g_0p_{t,\gamma}  n)\,d\mu_{g_0p_{t,\gamma}N^+,\nu_1}{}(n)
\\=&\,
\sum_{\gamma \in \Gamma_{t}
} \int_{N_R^+} (\tilde{F}+\tilde h) (g_0p_{t,\gamma}  n)\phi_{\epsilon_0 e^{-\alpha t}}^+\big(a_t\,n_{t,\gamma}^{-1}\;a_t^{-1}
\big)\,d\mu_{g_0p_{t,\gamma}N^+,\nu_1}{}(n).
\end{align*}
Similarly as before, we have, for all $t>t_0 $ and $\,n\in N_R^+$,
\begin{align}\label{eq.x2}
\phi_{\epsilon_0 e^{-\alpha t}}^+\big(a_t
\,n_{t,\gamma}^{-1}\;a_t^{-1}
\big)&=\phi_{\epsilon_0 e^{-\alpha t}}^+\big(a_t\,n_{t,\gamma}^{-1}n(n)^{-1}\;a_t^{-1}\big)\notag\\
&\leq \phi_{(R+\epsilon_0) e^{-\alpha t}}^+\big(a_t
\,n_{t,\gamma}^{-1}n\;a_t^{-1}\big).
\end{align}
Hence \eqref{intchain} is bounded above by
\begin{align*}
&\leq 
\begin{multlined}[t]
\sum_{\gamma \in \Gamma_{t}
} \int_{N_R^+} (\tilde{F}+\tilde h) (g_0p_{t,\gamma}  n)\phi_{(R+\epsilon_0) e^{-\alpha t}}^+\big(a_t\,n_{t,\gamma}^{-1}n\;a_t^{-1}
\big)\,d\mu_{g_0p_{t,\gamma}N^+,\nu_1}{}(n)
\end{multlined}
\\&=
\begin{multlined}[t]
\sum_{\gamma \in \Gamma_{t}
} \int_{N^+} (\tilde{F}+\tilde h) \big(g_0p_{t,\gamma} n_{t,\gamma}a_t^{-1} n a_t\big)\phi_{(R+\epsilon_0) e^{-\alpha t}}^+(n) \,
d((\theta_{t,\ga})^{-1}_*\mu_{g_0p_{t,\gamma}N^+,\nu_1})(n)
\end{multlined}
\end{align*}
where $\theta_{t,\ga}(n)=n_{t,\ga}a_t^{-1}na_t$.
By Lemma \ref{MANconj},
\begin{equation*}
d( (\theta_{t,\ga})^{-1}_*\mu_{g_0p_{t,\gamma}N^+,\nu_1})(n)
=e^{\psi_1(\log  a_t)}d\mu_{g_0p_{t,\gamma} n_{t,\gamma}a_t^{-1} N^+,\nu_1}{}(n).
\end{equation*}
Since  $g_0p_{t,\gamma}n_{t,\gamma}a_t^{-1}=\gamma g$, it follows that for all $t>t_0$,
\begin{align*}
 & e^{(\varphi_1-\psi_1)(\log a_t)} \int_{N^+} f([g] n a_{t})\phi(n)\,d\mu_{\lambda_1}{}(n)\,  \\ \le 
  &\sum_{\gamma \in \Gamma_{t}} \!\int_{N^+} (\tilde{F}+\tilde h) (\gamma g na_{t} )\phi_{(R+\epsilon_0) e^{-\alpha t}}^+(n)\,d\mu_{\gamma gN^+,\nu_1}{}(n)\\
\leq&\, \int_{N^+}\left(\sum_{\gamma \in \Gamma} (\tilde{F}+\tilde h) (\gamma g na_{t} )\right)\phi_{(R+\epsilon_0) e^{-\alpha t}}^+(n)\,d\mu_{\nu_1}{}(n).
\end{align*}
Define functions $F$ and $h$ on $\GaG$ by
\begin{equation*}
F([g]):= \sum_{\gamma\in\Gamma} \tilde{F}(\gamma g)\quad \text{ and } \quad h([g]):= \sum_{\gamma\in\Gamma} \tilde{h}(\gamma g).
\end{equation*}
Then for any $\epsilon>0$ and for all $t>t_0$ such that $(R+\epsilon_0)e^{-\alpha t} \leq \epsilon$,
\begin{multline*}
\Psi(t) e^{(\varphi_1-\psi_1)(\log a_t)}\int_{N^+}f([g] n a_{t})\phi(n)\,d\mu_{\lambda_1}{}(n)
\\ \leq \, \Psi(t)\int_{N^+} (F+h)([g]n  a_t )\phi^+_{\epsilon}(n)\,d\mu_{\nu_1}{}(n).
\end{multline*}
By Proposition \ref{p1}, letting $\epsilon\rightarrow0$ gives
\begin{multline*}
\limsup_{t\rightarrow +\infty} \Psi(t) e^{(\varphi_1-\psi_1)(\log a_t)}\int_{N^+}f([g]n  a_{t} )\phi(n)\,d\mu_{\lambda_1}{}(n)\\ \leq \sum_{Z\in
\frak Z_\Ga}
{\mathsf m}|_Z(F+h)\,\mu_{\nu_1}{}|_{ZN}(\phi).
\end{multline*}

Note that
$\mathsf m^*=\mathsf m$ by Lemma \ref{dudu}.
Now, by  Lemma \ref{lem.DC} and the fact $\tilde{\mathsf{m}}(\tilde h)=0$, we have 
\begin{align*}
&{\mathsf m}|_Z(F+h)=\tilde{\mathsf m}|_{\tilde Z}(\tilde{F}+\tilde h)=\tilde{\mathsf m}|_{\tilde Z}(\tilde{F})
=\tilde{\mathsf m}^*|_{\tilde Z}(\tilde{F})
\\
&=\int_{ P} \left(\int_{N^+} \tilde{F}\mathbbm{1}_{\tilde Z}(g_0hamn)\,d\mu_{g_0 hamN^+,\nu_1}^{}(n)\right)e^{-\psi_2\circ\i(\log a)}\,dm\,da\,d\mu_{g_0N,\nu_2}{}(h)
\\&=\int_{ P} \left(\int_{N^+} (\tilde{f}\mathbbm{1}_{ZN^+})(g_0hamn)\,d\mu_{g_0hamN^+,\lambda_1}{}(n)\right) e^{-\psi_2\circ\i(\log a)}\,dm\,da\,d\mu_{g_0N,\nu_2}{}(h)
\\
&=\tilde{m}_{\lambda_1,\nu_2}{}|_{\tilde ZN^+}(\tilde{f})=m_{\lambda_1,\nu_2}{}|_{ZN^+}(f).
\end{align*}

This gives the desired upper bound.
Note that we have used the assumption $\supp\nu_2=\La$ in the fourth equality above to apply \eqref{eq.1}.
The lower bound can be obtained similarly, finishing the proof.
\end{proof}

With the help of Proposition \ref{p1}, we are now ready to give:

\noindent{\bf Proof of Theorem \ref{prop.mixH000}}
By the compactness hypothesis on the supports of $f_i$, we can find $\epsilon_0>0$ and $x_i\in\Ga\ba G$, $i=1,\cdots , \ell$ such that the map $G\to\Ga\ba G$ given by $g\to x_ig$ is injective on $R_{\epsilon_0}=P_{ \epsilon_0}N_{ \epsilon_0}^+$, and $\bigcup_{i=1}^\ell x_i R_{\epsilon_0/2}$ contains both $\supp f_1$ and $\supp f_2$. 
We use continuous partitions of unity to write $f_1$ and $f_2$ as finite sums $f_1=\sum_{i=1}^\ell f_{1,i}$ and $f_2=\sum_{j=1}^\ell f_{2,j}$ with $\supp f_{1,i}\subset x_i R_{\epsilon_0/2}$  and $\supp f_{2,j}\subset x_j R_{\epsilon_0/2}$.
Writing $p=ham\in NAM$ and using Lemma \ref{lem.DC},
$$dm_{\lambda_1,\lambda_2}^*(hamn)=d\mu_{hamN^+,\lambda_1}{}(n)e^{-\psi_2\circ\i(\log a)}\,dm\,da\,d\mu_{N,\lambda_2}{}(h).$$

We have
\begin{align}\label{eq.pu1}
&\int_{\GaG} f_1 (x
a_t
)f_2(x)\,dm_{\lambda_1,\lambda_2}^*(x)=\\
&\sum_{i,j} \int_{R_{\e_0}} f_{1,i} (x_jpna_t)f_{2,j} (x_jpn )d\mu_{hamN^+,\lambda_1}{}(n)e^{-\psi_2\circ\i(\log a)}\,dm\,da\,d\mu_{N,\lambda_2}{}(h)\notag
\\&=\sum_{i,j} \int_{N_{\e_0}A_{\e_0}M_{\e_0}}\left( \int_{N_{\e_0}^+} f_{1,i} (x_j pna_t)f_{2,j} (x_jpn )\,d\mu_{hamN^+,\lambda_1}{}(n)\right)\notag
 \\
 &\qquad\qquad\qquad\qquad\qquad\qquad\qquad\qquad \times e^{-\psi_2\circ\i(\log a)}\,dm\,da\,d\mu_{N,\lambda_2}{}(h).\notag
\end{align}
Applying Proposition \ref{p2}, it follows:

\begin{multline*}
\lim_{t\rightarrow\infty} \Psi(t)
e^{(\varphi_1-\psi_1)(\log a_t)} \int_{\GaG} f_1 (x a_t)f_2(x)\,dm_{\lambda_1,\lambda_2}^*(x)
\\ = \sum_j\sum_{ Z\in\frak Z_\Ga} m _{\lambda_1,\nu_2}{}|_{ZN^+}(f_{1,j})\sum_i\int_{N_{\e_0}A_{\e_0}M_{\e_0}} \mu_{x_ipN^+,\nu_1}{}|_{ZN} (f_{2,i}(x_j p\,\cdot\,) )
\\ e^{-\psi_2\circ\i(\log 
a)}\,dm\,da\,d\mu_{N,\lambda_2}{}(h)
\\=\sum_{Z\in\frak Z_\Ga} m _{\lambda_1,\nu_2}|_{ZN^+}(f_{1})\sum_i\int_{N_{\e_0}A_{\e_0}M_{\e_0}} \mu_{x_ipN^+,\nu_1}{}(f_{2,i}\mathbbm{1}_{ZN}(x_j p\,\cdot\,) )\\
e^{-\psi_2\circ\i(\log 
a)}\,dm\,da\,d\mu_{N,\lambda_2}{}(h)
\\=\sum_{Z\in\frak Z_\Ga} m _{\lambda_1,\nu_2}{}|_{ZN^+}(f_{1})\sum_i m _{\nu_1,\lambda_2}^*{}(f_{2,i}\mathbbm{1}_{ZN})=\sum_{Z\in\frak Z_\Ga}m _{\lambda_1,\nu_2}|_{ZN^+}(f_{1})  m _{\nu_1,\lambda_2}^*|_{ZN}(f_{2}) 
\end{multline*}
where the second last equality is valid by 
Lemma \ref{lem.DC}.
This completes the proof.$\qed$

\section{Local mixing for Anosov groups}\label{sec:lo}
Let $\Ga<G$ be a Zariski dense Anosov subgroup with respect to $P$.
For any $u\in \inte\L_\Ga$, there exists a unique $$\psi=\psi_u\in D_\Ga^\star$$ such that $\psi(u)=\psi_\Ga(u)$ \cite[Prop. 4.4]{LO}.
Let $\nu_\psi$ denote the unique $(\Ga,\psi)$-$\PS$ measure \cite[Thm. 1.3]{LO}. Similarly, $\nu_{\psi \circ \i}$ denotes the unique $(\Ga, \psi \circ \i)$-$\PS$-measure.

In this section, we deduce ($r:=\op{dim}\frak a$):
\begin{Thm}[Local mixing]  \label{thm.GE} 
For $i=1,2$, let $\varphi_i\in\fa^*$ and $\lambda_{\varphi_i}$ be any $(\Ga,\varphi_i)$-conformal measure on $\cal F$. For any  $u\in \inte\L_\Ga$, there exists $\kappa_u>0$ such that
for any $f_1, f_2\in C_c(\Gamma\ba G)$, we have
\begin{multline*}
\lim_{t\rightarrow +\infty} t^{(r-1)/2}e^{(\varphi_1-\psi_u)(tu)}  \int_{\GaG} f_1 (x \exp(tu))f_2(x)\,dm^*_{\lambda_{\varphi_1},\lambda_{\varphi_2}}(x)
\\ =\kappa_u \sum_{Z\in\frak Z_\Ga}\,m_{\lambda_{\varphi_1},\nu_{\psi_u\circ\i}}|_{ZN^+}(f_1)\,m_{\nu_{\psi_u},\lambda_{\varphi_2}}^*|_{ZN}(f_2).
\end{multline*}
\end{Thm}

Theorem \ref{thm.GE} is a consequence of Theorem \ref{prop.mixH000}, 
since the measure $\mathsf m=m_{\nu_{\psi_u},\nu_{\psi_u\circ\i}}$ satisfies the Hypothesis \ref{eq.hypo} by the following theorem of Chow and Sarkar.
\begin{thm} \cite{CS}\label{CS}  Let $u\in \inte\L_\Ga$. There exists $\kappa_u>0$
such that
for any $f_1, f_2\in C_c(\Gamma\ba G)$, we have
\begin{multline*}
\lim_{t\rightarrow +\infty} t^{(r-1)/2}  \int_{\GaG} f_1 (x \exp (tu) )f_2(x)\,dm_{\nu_{\psi_u},\nu_{\psi_u\circ\i}}{}(x)\\
=\kappa_u \sum_{Z\in\frak Z_\Ga}\,m_{\nu_{\psi_u},\nu_{\psi_u\circ\i}}{}|_Z(f_1)\, m_{\nu_{\psi_u},\nu_{\psi_u\circ\i}}{}|_Z(f_2).
\end{multline*}
\end{thm}

Let $m_o$ denote the $K$-invariant probability measure on $\F=G/P$.
Then  $m_o$ coincides with the
$(G, 2\rho)$-conformal measure on $\F$ where
$2 \rho$ denotes the sum of positive roots for $(\frak g, \fa^+)$. The corresponding BMS measure $dx=dm_{m_o,m_o}$ is
a $G$-invariant measure on $\Ga\ba G$. 
The measure $dm_{\nu_{\psi\circ \i}}^{\BR}=dm_{m_o,\nu_{\psi\circ \i}}$  was defined and  called the
 $N^+M$-invariant Burger-Roblin measure in \cite{ELO}. Similarly, the $NM$-invariant Burger-Roblin measure was defined as
 $dm^{\BR*}_{\nu_{\psi}} $. In these terminologies,  the following is a special case of Theorem \ref{thm.GE}:
\begin{cor}[Local mixing for the Haar measure]\label{above} For any $u\in \inte\L_\Ga$,
and for
 any $f_1, f_2\in C_c(\Gamma\ba G)$, we have
\begin{multline*}
\lim_{t\rightarrow +\infty} t^{(r-1)/2}e^{(2\rho-\psi_u)(tu)}  \int_{\GaG} f_1 (x \exp(tu))f_2(x)\,dx
\\ =\kappa_u \sum_{Z\in\frak Z_\Ga}\,m_{\nu_{\psi_u\circ \i}}^{\BR}|_{ZN^+}(f_1)\, m_{\nu_{\psi_u}}^{\BR*}|_{ZN}(f_2)
\end{multline*}
where $\kappa_u$ is as in Theorem \ref{CS}.
\end{cor}

In fact, we get the following more elaborate version of the above corollary by combining the proof of \cite[Theorem 7.12]{ELO} and
the proof of Corollary \ref{above}.
\begin{thm} \label{m11}\label{prop.mixH} Let $u\in \inte\L_\Ga$.
 For any $f_1, f_2\in C_c(\Ga\ba G)$ and   $v\in \ker \psi_u$,
 $$
\begin{multlined}[t]
\lim_{t\to +\infty} t^{(r-1)/2} e^{(2\rho -\psi_u) ( tu+\sqrt tv)}  \int_{\Gamma\ba G} f_1(x \exp(tu+\sqrt tv) ) f_2(x) dx 
\\
= \kappa_u\, e^{-I(v)/2}\sum_{Z\in\frak Z_\Ga}\,m_{\nu_{\psi_u\circ \i}}^{\BR}|_{ZN^+}(f_1)\, m_{\nu_{\psi_u}}^{\BR*}|_{ZN}(f_2)
\end{multlined}
$$
where $I : \op{ker}\psi_u\to \bb R$ is given by
\begin{equation}\label{ii}
I(v):= c \cdot
\frac{\|v\|_*^2 \|u\|_*^2 -\langle v, u\ra_*^2}{\|u\|_*^2}
\end{equation} for some inner product $\langle\cdot,\cdot \ra_*$ and some $c>0$.
Moreover the left-hand sides of the above equalities are uniformly bounded 
 for all $(t,v)\in\bb (0,\infty)\times\ker\psi_u$ with $tu+\sqrt tv\in\mathfrak a^+$. 
\end{thm}

\section{Proof of Theorem \ref{main}}
Let $\Ga<G$ be a Zariski dense Anosov subgroup with respect to $P$.

\subsection*{The $u$-balanced measures}

Let $\Omega=\{[g]\in \Ga\ba G: g^{\pm}\in \La\}$.
Following \cite{BLLO}, given $u\in \inte \L_\Ga$,
we say that
a locally finite Borel measure $\mathsf m_0$ on $\Gamma\ba G$ is $ u$-balanced if
$$
\limsup_{T\to+\infty}\frac{\int_0^T\mathsf m_0(\cal O_1\cap\cal O_1\exp(tu))\,dt}{\int_0^T\mathsf m_0(\cal O_2\cap\cal O_2\exp(tu))\,dt}<\infty,
$$
for all bounded $M$-invariant Borel subsets $\cal O_i\subset\Ga\ba G$ with $\Omega\cap \op{int}\cal O_i\ne \emptyset$, $i=1,2$.

As an immediate corollary of Theorem \ref{thm.GE}, we get
\begin{cor}\label{bal} Let $\varphi\in \fa^*$.
For any pair 
$(\lambda_{\varphi}$, $\lambda_{\varphi\circ \i})$ of
$(\Ga,\varphi)$ and $(\Ga,\varphi\circ \i)$-conformal measures on $\cal F$ respectively, the corresponding BMS-measure
$m_{\lambda_{\varphi},\lambda_{\varphi\circ \i}}$ is $u$-balanced for any $u\in \inte\L_\Ga$.
\end{cor}
\begin{proof} Let $\cal O_1,\cal O_2$ be $M$-invariant Borel subsets such that $\Omega\cap \inte \cal O_i\ne\emptyset$ for each $i=1,2$.
Let $f_1, f_2\in C_c(\Ga\ba G)$ be non-negative functions such that $f_1\ge 1$ on $\cal O_1$
and $f_2\le 1$ on $\cal O_2$ and $0$ outside $\cal O_2$.
Since $\inte \cal O_2\cap \Omega\ne \emptyset$, we may choose  $f_2$ so that $m_{\nu_{\psi_u},\lambda_{\varphi\circ\i}}^*( f_2)>0$. For simplicity, we set $\mathsf m_0= m_{\lambda_{\varphi},\lambda_{\varphi\circ \i}}$.
By Theorem \ref{thm.GE} and using the fact that $\mathsf m_0$ is $A$-quasi-invariant, we obtain that for any $u\in \inte\L_\Ga$,
\begin{align*}
&\limsup_{t\to+\infty}\frac{\mathsf m_0(\cal O_1\cap\cal O_1\exp(tu))}{\mathsf m_0(\cal O_2\cap\cal O_2\exp(tu))}  \\
&\le \limsup_{t\to+\infty}\frac{\int f_1 (x) f_1 (x\exp (-tu)) d\mathsf m_0(x) }{\int f_2 (x) f_2 (x\exp (-tu)) d\mathsf m_0 (x) }\\
&=\limsup_{t\to+\infty}\frac{\int f_1 (x) f_1 (x\exp (tu)) d\mathsf m_0(x) }{\int f_2 (x) f_2 (x\exp (tu)) d\mathsf m_0 (x) }\\
&=\limsup_{t\to+\infty}\frac{t^{(r-1)/2} e^{(\varphi-\psi_u)(tu)}\int f_1 (x) f_1 (x\exp (tu)) d\mathsf m_0(x) }{t^{(r-1)/2} e^{(\varphi-\psi_u)(tu)}\int f_2 (x) f_2 (x\exp (tu)) d\mathsf m_0 (x) }\\
&=\frac{m_{\lambda_{\varphi},\nu_{\psi_u\circ\i}}( f_1)}{m_{\nu_{\psi_u},\lambda_{\varphi\circ\i}}^*( f_2)} <\infty .
\end{align*}
This shows that $\mathsf m_0$ is $u$-balanced.
\end{proof}
Recall Theorem \ref{main} from the introduction:

\begin{thm} Let $\op{rank}G\leq 3$.
For any $\psi\in \dg$, any $(\Ga, \psi)$-conformal measure on $\F$
is necessarily supported on $\La$. Moreover, the
$\PS$ measure $\nu_\psi$ is
the unique $(\Ga, \psi)$-conformal measure on $\F$.
\end{thm}
\begin{proof} Let $u\in \inte\L_\Ga$ denote the unique unit vector such that $\psi(u)=\psi_\Ga(u)$, that is, $\psi=\psi_u$.
Let $\lambda_\psi$ be any $(\Ga, \psi)$-conformal measure on $\F$. We claim that $\lambda_\psi$ is supported on $\La$.
The main ingredient is the higher rank Hopf-Tsuji-Sullivan dichotomy established in \cite{BLLO}.
The main point is that all seven conditions
of Theorem 1.4 of \cite{BLLO} are equivalent to each other for Anosov groups
and $u\in \inte \L_\Ga$, since all the measures considered there are $u$-balanced by Corollary \ref{bal}. In this proof,
we only need the equivalence of (6) and (7), which we now recall.

Consider the following $u$-directional conical limit set of $\Ga$:
\be\label{lau} \La_u:=\{g^+\in\La: \ga_i\exp(t_iu)\text{ is bounded for some $t_i\to +\infty$ and $\ga_i\in \Ga$} \}.\ee 
Note that $\La_u\subset \La$.
For $R>0$, we set $\Ga_{u,R}:=\{\ga\in\Ga: \norm{\mu(\ga)-\br u}<R\}$. Applying the dichotomy \cite[Thm. 1.4]{BLLO} to a $u$-balanced measure $m{}_{\lambda_\psi, \nu_{\psi\circ\i}}$, we deduce
\begin{prop}\label{thm.Rob}\label{dio}
The following conditions are equivalent for  $\la_\psi$:
\begin{enumerate}
\item  $\lambda_{\psi}(\La_u)=1$;
\item 
$\sum_{\ga\in\Ga_{u,R}} e^{-\psi(\mu(\ga))}=\infty$ for some $R>0$.
\end{enumerate}

\end{prop}
On the other hand, if $\op{rank}G\leq 3$, we have
$$\sum_{\ga\in\Ga_{u,R}} e^{-\psi(\mu(\ga))}=\infty$$ for some $R>0$  \cite[Thm. 6.3]{BLLO}.
Therefore, by Proposition \ref{dio},
we have $\lambda_\psi(\La_u)=1$ and hence $\lambda_\psi$ is supported on $\La$ in this case.
This finishes the proof of the first part of Theorem \ref{main}.
The second claim follows from the first one by \cite[Thm. 1.3]{LO}.
\end{proof}

\end{document}